\documentclass[12pt]{amsart}
\usepackage{amsmath}
\usepackage{amssymb,amscd}
\usepackage[curve]{xypic}
\usepackage{graphicx}
\usepackage[colorlinks=true, urlcolor=blue,bookmarks=true,bookmarksopen=true, citecolor=blue]{hyperref}

\numberwithin{equation}{section}

\newtheorem{theorem}{Theorem}[section]

\newtheorem{corollary}[theorem]{Corollary}

\newtheorem{lemma}[theorem]{Lemma}
\newtheorem{prop}[theorem]{Proposition}

\theoremstyle{definition}
\newtheorem{defn}[theorem]{Definition}

\newtheorem{example}[theorem]{Example}
\newtheorem{ques}[theorem]{Question}
\newtheorem{remark}[theorem]{Remark}

\def \begineq{\begin{equation}}
\def \endeq{\end{equation}}

\def \bb{\mathbb}

\def \CC{{\bb{C}}}

\def \QQ{{\bb{Q}}}
\def \RR{{\bb{R}}}

\def \ZZ{{\bb{Z}}}

\def \({\left(}
\def \){\right)}
\def \<{\langle}
\def \>{\rangle}
\def \bar{\overline}

\def \tensor{\otimes}

\begin{document}

\title{Cohomology Rings of a Class of Torus Manifolds}
\author[S. Sarkar]{Soumen Sarkar}

\address{Department of Mathematics and Statistics, University of Regina, Regina, Canada}

\email{soumen@iitm.ac.in}

\author[D. Stanley]{Donald Stanley}

\address{Department of Mathematics and Statistics, University of Regina, Regina, Canada}

\email{Donald.Stanley@uregina.ca}

\subjclass[2010]{55N99, 57R99, 52B05}

\keywords{polytopes, torus action, torus manifold, (equivariant) connected sum,
homology groups, cohomology ring}

\thanks{}

\abstract Torus manifolds are topological generalization of smooth projective
toric manifolds. We compute the rational cohomology ring of a class of smooth
locally standard torus manifolds whose orbit space is a connected sum of simple polytopes. 
\endabstract

\maketitle

\section{Introduction}\label{intro}
Topological generalization of non-singular smooth projective toric varieties first
appeared in the pioneer work of Davis and Januszkiewicz \cite{DJ} with the name `toric 
manifold'. These spaces are also known as quasitoric manifolds \cite[Chapter 5]{BP}.
Further generalizations like torus manifolds,
manifolds with local torus action and topological torus manifolds have appeared
in \cite{HM},  \cite{Yo} and \cite{IFM}. The study of torus actions
on manifolds and orbifolds sparks many interesting developments.
Briefly, A torus manifold is an even dimentional manifold with an effective
half-dimensional torus action having non-empty fixed points. In addition,
if the local representation resembles the standard action, then it is known
as a locally standard torus manifold. An example of a locally standard torus
manifold which is not toric manifold is even dimensional sphere $S^{2n}$ with the
standard action of $T^n$, see Example \ref{tormfd}. 
Many topological invariants of locally standard torus manifolds are known,
for example \cite{MP, AMPZ}, \cite{Ayz1, Ayz2}, but mainly they consider
the corresponding orbit space is face acyclic or the proper faces of the 
orbit space are acyclic. The face acyclic condition on the orbit space leads to 
trivial cohomology in odd-degrees of the torus manifold. \cite[Theorem 4.1]{MP} says that
if the odd-degree cohomology of a torus manifold is trivial then the torus action is 
locally standard. 

In this article we investigate the cohomological properties
of torus manifolds where the associated orbit space may have non-acyclic
proper faces. We consider torus manifolds which may be obtained by equivariant
gluing of toric manifolds along deleted neighborhood of nontrivial torus
orbits, see Section \ref{def}. We remark that more general connected sums may
be found in \cite{GK}. In this section, we show that if the orbit space of a
locally standard torus manifold $M$ is a connected sum of simple polytopes
then $M$ is an equivariant connected sum of toric manifolds, see Lemma \ref{lem_conn_sum1}.
 Moreover, in Lemma \ref{lem_homeo_clasi} we show that an equivariant connected sum
 of two toric manifolds is homemorphic to a connected sum of two toric manifolds and
 one locally standard torus manifold. In Section \ref{tcev}, we study the homotopy type of the
complement of an orbit of the natural $T^n$-action on $S^{2n}$. We show here that
an orbit complement of $S^{2n}$ is homotopic to a wedge of some lower dimensional
spheres. Our equivariant gluing is not at a fixed point. That means the torus manifolds
considered here have nontrivial homology in odd-degrees, see Section \ref{sec:hom_gps}.
In this section we compute the integral homology of equvariant 
connected sum $S^{2n} \#_{T^k} S^{2n}$ of two $S^{2n}$ along a non-trivial orbit
$T^k \hookrightarrow S^{2n}$. Section \ref{sec:cohom_ring} contains the following
main theorem where $\mathbb{K}$ is a field of characteristic zero.
\begin{theorem}\label{thm_conn_sum_s2n}
 Let
 \begin{equation}
  A = s^{-1}\widetilde{H}^{\ast}T^k \oplus s^{-2n+1}(\widetilde{H}^{\ast}T^k)^{\#} \oplus \mathbb{K} 1 \oplus \mathbb{K} \mu_{2n}
 \end{equation}
be a commutative differential graded algebra with the differential $d=0$ and
a multiplication $\beta \colon A \tensor A \to A$ which is determined by 
\begin{equation}\label{eq:mult_in_A}
\begin{array}{l}
\beta(s^{-1}\alpha_I, s^{-1}\alpha_J) = 0, 
 ~~\beta(s^{-1}\alpha_I, s^{-2n+1}\alpha_J^{\#}) = -\delta_{IJ} \mu_{2n}, ~~\beta(1, x) =x, \\
\beta(s^{-2n+1}\alpha_I^{\#}, s^{-2n+1}\alpha_J^{\#}) = 0, ~\mbox{and} ~ \beta(\mu_{2n}, y)=0 ~ \mbox{for} ~ y \neq 1, ~\mbox{with}~ I, J \subseteq \{1, \ldots, k\}.
\end{array}
\end{equation}
where $\{\alpha_I\}$ is a graded basis for $\widetilde{H}^{\ast}T^k$ for
$I \subseteq \{1, \ldots, k\}$. 
Then $A$ is a model for the torus manifold $S^{2n} \#_{T^k} S^{2n}$. 
\end{theorem}
The proof of this theorem follows from Propositions \ref{30a} and \ref{30b}.
At the end of this section we give a presentation
of the cohomology ring of locally standard torus manifolds over connected sum of simple
polytopes, see Theorem  \ref{thm:cohom_conn_sum}.


\section{Basics of torus manifolds}\label{def}

\subsection{Connected sum of manifolds with corners}\label{sec:conn_sum}
A polytope is the convex hull of a finite set of points in
$\RR^n$. An $n$-dimensional polytope is said to be simple if every
vertex is the intersection of exactly $n$ codimension one faces.
Ready examples of simple polytopes are simplices and cubes.
Manifold with corners and manifold with faces are discussed in detail in Section 6 of \cite{Da}.
Here by manifold with corners we will refer nice manifold with corners.
A class of examples of manifolds with corners are simple polytopes.
More properties on manifold with corners can be found in \cite{Jo}.

Now we discuss connected sum of manifolds with corners at a relative interior point of faces.
Let $P$ and $Q$ be two $n$-dimensional oriented manifolds with corners in $\RR^n$.
Let $F$ and $G$ be $k$-dimensional ($k \geq 0$) faces of $P$ and $Q$
respectively. Let $x_F$ and $y_G$ belong to the relative interior of
$F$ and $G$ respectively. Let $B$ and $C$ be open $n$-ball around
$x_F$ and $y_G$ respectively in $\RR^n$ such that
\begin{equation}\label{con}
 \begin{array}{ll} (1) ~B \cap P ~\mbox{and} ~ C \cap Q ~ \mbox{are homeomorphic as manifold with corners to} ~\RR^k \times \RR^{n-k}_{\geq 0}.\\
 (2)~ \bar{B} \cap H ~\mbox{is empty for any face} ~ H ~ \mbox{of} ~ P ~\mbox{not containing} ~F.\\
 (3) ~\bar{C} \cap K ~\mbox{is empty for any face} ~ K~ \mbox{of} ~ Q ~ \mbox{not containing} ~ G.
 \end{array}
\end{equation}

Then $\partial{\bar{B}} \cap P$ and $\partial{\bar{C}} \cap Q$ are homeomorphic
as manifold with corners. Identifying $\partial{\bar{B}} \cap P$ and 
$\partial{\bar{C}} \cap Q$ via an oriented reversing homeomorphism,
we get a new oriented manifold with corners, denoted by $P \#_{x_F, y_G} Q$.

If $P, Q$ are simple polytopes and $F, G$ are vertices of $P, Q$ respectively,
then the connected sum $P \#_{x_F, y_G} Q$ is homeomorphic as manifold with corners
to a simple polytope, see Construction 1.13 in \cite{BP}.

If $P, Q$ are simple polytopes and $F = P$ and $G = Q$,
then the connected sum $P \#_{x_F, y_G} Q$ is diffeomorphic as manifold with
corners to a polytope with simple holes, see Subsection 2.1 in \cite{PS2}.

\begin{example}
Some polytope with simple holes in $\RR^2$ are given in Figure \ref{pol_hol}.
The first figure is the connected sum of a rectangle and a triangle at
their interior points. The second figure is the connected sum of an octagon,
a rectangle and a triangle at their interior points. 
\begin{figure}[ht]
      \centerline{
      \scalebox{0.70}{
   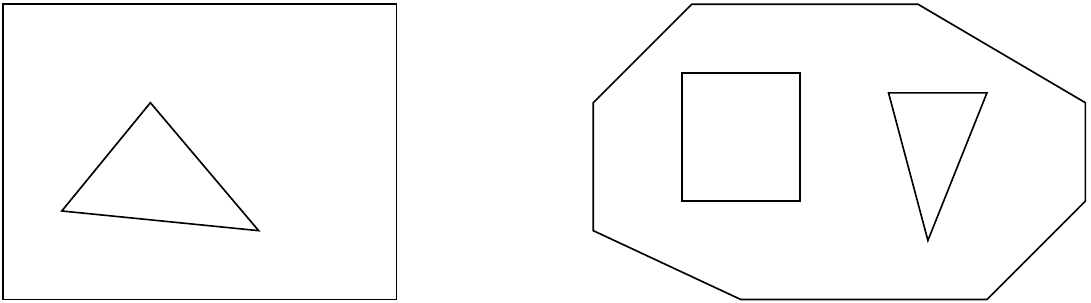
          }
     }
  \caption {Polytopes with simple holes in $\RR^2$.} 
  \label{pol_hol}
\end{figure}
\end{example}

\begin{remark}
$\partial{\bar{B}} \cap P$ is contractible if and only if dimension of $F$ is less than $n$.
So $P \#_{x_F, y_G} Q$ is contractible if and only if dimension of $F$ is less than $n$
and $P, Q$ are contractible.
\end{remark}

\begin{lemma}
Let $P$ be an $n$-dimensional polytope with simple holes. Then $H^2(P, \ZZ)$ is trivial
 if and only if $n \neq 3$.
\end{lemma}
\begin{proof}
Let $P$ be an $n$-dimensional polytope with $s$ simple holes. So $P$ is homotopic
to wedge of $s$ many $(n-1)$-dimensional sphere. Lemma follows from this.
\end{proof}

Now we discuss another type of connected sum of manifolds with corners at faces.
Let $P$ and $Q$ be two $n$-dimensional oriented manifolds with corners in $\RR^n$.
Let $F$ and $G$ be $k$-dimensional ($k \geq 1$) contractible faces of $P$ and $Q$
respectively. Let $B$ and $C$ be tubular neighborhood of $F$ and $G$ in $P$
and $Q$ respectively such that
\begin{enumerate}
 \item  $ P - B$ and $Q - C$ are manifold with corners in $\RR^n$.
\item $B$ and $C$ are homeomorphic as manifold with corners
to the sets $F \times \RR^{n-k}_{\geq 0}$ and $G \times \RR^{n-k}_{\geq 0}$ respectively.
\end{enumerate}
Then $\bar{B} \cap (P-B)$ and $\bar{C} \cap (Q-C)$ are homeomorphic as
manifold with corners. Identifying
$\bar{B} \cap (P-C)$ and $\bar{C} \cap (Q-C)$ via a homeomorphism, we get a new
oriented manifold with corners, denoted by $P \#_{F, G} Q$.
Using Van-Kampen theorem for fundamental group, one can show the following.
\begin{lemma}
If $P$ and $Q$ are contractible manifold with corners, then $P \#_{x_F, y_G} Q$
is contractible. In addition, if $F, G$ are contractible then $P \#_{F, G} Q$ 
is contractible manifolds with corners.
\end{lemma}

\subsection{Combinatorial construction}
 Let $P$ be an $n$-dimensional oriented manifold with corners in $\RR^n$.
Let $ \mathcal{F}(P) = \{ F_1, F_2,\ldots, F_m \} $ be the
set of all codimension one faces ({\em facets}) of $P$. Note that if
$F$ is a nonempty face of $P$ of codimension $k$ then $F$ is a connected component of the
intersection of a unique collection of $k$ facets of $P$.
The following definition is a straightforward generalization of the notion
of characteristic function for a simple polytope, which is a crucial concept for
studying toric manifolds \cite{DJ,BP}.

\begin{defn}\label{def:tm}\cite{HM} 
A closed, connected, oriented, smooth manifold $Y$ of
dimension $2n$ with an effective smooth action of $T^n$ with
non-empty fixed point set is called a torus manifold if a
preferred orientation is given for each characteristic
submanifold. A characteristic submanifold is, by definition, any
codimension two closed connected submanifold of $Y$, which is
fixed by some circle subgroup of $T^n$ and contains at least one
$T^n$-fixed point.
\end{defn}

\begin{defn} \cite{HM}
 A torus manifold $M$ of dimension $2n$ is locally standard if every point in $M$
has an invariant neighborhood $U$ weakly equivariantly diffeomorphic to an open
subset $W \subset \CC^n$ invariant under the standard $T^n$-action on $\CC^n$.
\end{defn}

We briefly recall the construction of some locally standard torus manifolds.
\begin{defn} A function
$ \xi \colon \mathcal{F}(P) \rightarrow \ZZ^n $  is called a
 characteristic function on $P$ if it satisfies the following
condition: Whenever $F = \bigcap_{j=1}^{k} F_{i_j} $ is an
$(n-k)$-dimensional face of $P$, the span of the vectors $ \xi(F_{i_1}),
\xi (F_{i_2}),\ldots, \xi(F_{i_k}) $ is a
$k$-dimensional direct summand of $\ZZ^n$. We will denote
$\xi(F_i)$ by $\xi_i$ for simplicity and call it the
characteristic vector of $F_i$.
\end{defn}

For any face $F = \bigcap_{j=1}^{k} F_{i_j}$ of $P$, let $N(F)$ be
the submodule of $\ZZ^{n}$ generated by  $\{ \xi_{i_1}, \ldots, \xi_{i_k} \}$.
 The module $ N(F) $ defines a sub-torus $T_F$ of $T^n = \ZZ^n
\otimes \RR /\ZZ^n =\RR^n/\ZZ^n $ as follows.
\begin{equation} T_F := (N(F) \otimes \RR) /N(F).
\end{equation}

Define an equivalence relation $ \sim $ on the product space $ T^n \times P $ by
\begin{equation}\label{ereln}
 (t,x) \sim (u,y) ~\mbox{if}~  x=y ~\mbox{and} ~ u^{-1}t \in T_{F}
\end{equation}
where $F$ is the unique face of $P$ whose relative interior contains $x$.

We denote the quotient space as follows.
 \begin{equation}\label{quo}
  M(P,\xi) := (T^n \times P)/\sim.
\end{equation}

 The space $M(P, \xi)$ is a $2n$-dimensional locally standard torus manifold.
 The proof of this is analogous to the toric manifold case in \cite{DJ}.
The  $T^n$ action on $( T^n \times P )$ induces a natural
effective action of $T^n$ on $M(P, \xi)$, which is locally standard.
 Let $ \pi \colon  M(P, \xi) \rightarrow P $  be the projection
or orbit map defined by $ \pi ([ (t,x) ]) = x $. Since $M(P, \xi)$ 
is locally standard, the $T^n$-fixed point set corresponds
bijectively to the set of vertices of $P$. Observe that the spaces
$X_i := \pi^{-1}(F_i)$,  $i = 1, \ldots, m, $ are the
characteristic submanifolds of $M(P, \xi)$. Each $X_i$ is a
$2(n-1)$-dimensional manifold. We say that $M(P, \xi)$ is the torus
manifold derived from  the {\em characteristic pair} $(P, \xi)$.

Let $M$ be a locally standard torus manifold and $P = M/T^n$ be the orbit space of
$T^n$-action on $M$. So $P$ is a manifold with corners and vertices of $P$ correspond
to the $T^n$-fixed points of $M$.  Let $\mathfrak{q} \colon M \to P$ be the quotient projection.
Let $\mathcal{F}(P)=\{P_1, \ldots, P_m\}$ be the facets of $P$. Then the characteristic
submanifolds (codimension-two submanifolds fixed pointwise by a circle subgroup of $T^n$)
of $M$ are $\{M_i = \mathfrak{q}^{-1}(P_i) : i= 1, \ldots, m\}.$
So we can choose a map
\begin{equation}
 \lambda \colon \mathcal{F}(P) \to \ZZ^n
\end{equation}
such that $\lambda(P_i)$ is the primitive vector in $\ZZ^n$ and determines the circle
subgroup of $T^n$ fixing $M_i$.

\begin{lemma}\cite{MP}
 If $P_{i_1} \cap \cdots \cap P_{i_k}$ is non-empty, then $\lambda(P_{i_1}), \ldots, \lambda(P_{i_k})$
is a part of a basis of $\ZZ^n$.
\end{lemma}

So $\lambda$ is a characteristic function on $P$. Let $M(P, \lambda)$ be the torus manifold
obtained by construction from the pair $(P, \lambda)$.
\begin{lemma}\cite{MP}
 If $H^2(P, \ZZ)=0$, then there is an equivariant homeomorphism $M(P, \lambda) \to M$ covering the identity on $P$.
\end{lemma}

\subsection{Equivariant connected sums}\label{equi_cunn_sum}
In Section 6 of \cite{GK} equivariant connected sum of two smooth manifolds with $T^n$-action is discussed explicitly.
Equivariant connected sum of toric manifolds is discussed in Section 6 of \cite{BR}.
We briefly recall the argument for equivariant connected sum of torus manifolds
in the following. Let $M_1$ and $M_2$ be two smooth locally standard torus manifolds
of dimension $2n$. Let $T_1 \subset M_1$ and $T_2 \subset M_2$ be two orbit of 
same dimension. Now changing the action (if necessary) of $T^n$ on $M_2$
by an automorphism of $T^n$, we may assume that $T_1$ and $T_2$ have same
stabilizer $G$ and for which the isotropy actions are isomorphic. Then by the
slice theorem each orbit $T_i$ has a tubular neighborhood $V_i$ equivariantly
diffeomorphic to $T^n \times_{H} D^{2\ell}$, where $D^{2\ell}$ is a disc in a linear
$G$-representation $\RR^{2\ell}$ for some $\ell$. Identifying the deleted
neighborhood $V_1 - T_1$ and $V_2 -T_2$ via an orientation reversing equivariant
diffeomorphism we get a smooth manifold, denoted by $M_1 \#_{T^k} M_2$, with
a natural locally standard $T^n$-action. If equivariant homeomorphism is
considered, then we may loose smooth structure. 

The orbit space of $T^n$-action on $M_1 \#_{T^k} M_2$ can be described as follows.
Let $\mathfrak{q}_1 \colon M_1 \to Q_1$ and $\mathfrak{q}_2: M_2 \to Q_2 $ be the
orbit maps for $T^n$-action on $M_1$ and $M_2$ respectively. Let
$\mathfrak{q}_1(T_1)=x_1 \in Q_1$ and $\mathfrak{q}_2(T_2)=x_2 \in Q_2$.
So $x_1$ and $x_2$ belong to the relative interior of $k$-dimensional face $H_1$
and $H_2$ of $Q_1$ and $Q_2$ respectively. Delete a neighborhood
$U_1$ of $x_1$ in $Q_1$ such that the closure of $U_1$ is homeomorphic as
manifold with corners to
$$\{(x_0, \ldots, x_n) \in \RR^n : x_0^2+ \cdots + x_n^2 \leq 1 ~\mbox{and}~ x_i\geq 0 ~\mbox{if}~i \geq k\} ~ \mbox{and}~ \bar{U} \cap F = \emptyset$$
for any face $F$ of $Q_1$ with $F \cap {H_1}^0 = \emptyset$. 
Let $Q^{\prime}_1$ be the remaining manifold with corners. Then $Q_1^{\prime}$ has
a new facet $\widetilde{H}_1$ which is homeomorphic as manifold with corners
to $$\{(x_0, \ldots, x_n) \in \RR^n ~:~
x_0^2+ \cdots + x_n^2 = 1 ~\mbox{and}~ x_i \geq 0 ~\mbox{if}~i \geq k\}.$$
Similarly we can construct the manifold with corners $Q_2^{\prime}$ from $Q_2$.
Let $H_i$ be a connected component of $F_1^i \cap \ldots \cap F_{n-k}^{i}$ for
some unique facets $F_1^i, \ldots, F_{n-k}^i$ of $Q_i$. By assumption on the
stabilizer of $T_1, T_2$, we may assume that the characteristic vector of $F_j^1$ and
$F_j^2$ are same for $j=1, \ldots, n-k$. From Subsection \ref{sec:conn_sum}
 one can construct a manifold with corners, denoted by $Q_1 \#_{x_1,x_2} Q_2$,
 by gluing $Q_1^{\prime}$ and $Q_2^{\prime}$ at $x_1$ and $x_2$ such that
 $F_j^1 \#{x_1, x_2} F_j^2$ makes a new facet for $j=1, \ldots, n-k$.
Then $M_1 \#_{T^k} M_2$ is a locally standard torus manifold over $Q_1 \#_{x_1, x_2} Q_2$.

\begin{example}\label{tormfd}
 Let $[n] = \{1, \ldots, n\}$ for all positive integer $n$. Consider
\begin{equation}
 S^{2n} = \{(z_1, \ldots, z_n, x) \in \CC^n \times \RR : |z_1|^2 + \cdots + |z_n| +x^2 =1\}.
\end{equation}
The natural $T^n$-action on $S^{2n}$ is defined by 
\begin{equation}
 ((t_1, \ldots, t_n) \times (z_1, \ldots, z_n, x)) \to (t_1z_1, \ldots, t_nz_n, x).
\end{equation}
This is a locally standard action and the orbit map is defined by 
$$\pi(z_1, \ldots, z_n, x) \to (|z_1|, \ldots, |z_n|, x).$$ So the orbit space is
$$Q^n:= \{(x_1, \ldots, x_n, x) \in \RR^{n+1} : x_1^2 + \cdots + x_n^2 +x^2=1 ~\mbox{and}~ x_i \geq 0 ~\mbox{for}~ i=1, \ldots, n\}.$$
Therefore, $Q^n$ is a nice manifold with corners. The vertices of $Q^n$ are $\{(0, \ldots, 0, \pm 1)\}$.
A codimension-$k$ ($0 < k < n$) face of $Q^n$ is given by 
$$\{(x_1, \ldots, x_n, x) \in Q^n : x_{i_1}=\cdots =x_{i_k}=0 ~\mbox{for some}~ \{i_1, \ldots, i_k\} \subset [n]\}.$$
In particular, the edges of $Q^n$ are given by $E_i:= \{(0, \ldots, 0, x_i, 0,
 \ldots, 0, x) \in Q^n\}$ for $i =1, \ldots, n$.
Let $F_i$ be the face of $Q^n$ defined by $$F_i:= \{(x_1, \ldots, x_{i-1}, 0,
 x_{i+1} \ldots, x_n, x) \in Q^n\} $$
for $i \in \{1, \ldots, n\}$. Then facets of $Q$ are $\{F_1, \ldots, F_n\}$.
 The subgroup which fixes
$\pi^{-1}(F_i)$ is $\{(0, \ldots, 0, t_i, 0, \ldots, 0) \in T^n\}$ for $i=1,
 \ldots, n$. Let $F$ be a  $k$-dimensional face of $Q$. Let $x_1, x_2$ be two
 points in the relative interior of $F$. Let $T_i=\pi^{-1}(x_i)$
for $i=1, 2$. Then $T_i$ satisfy the assumption in previous connected sum
construction. So we can construct locally standard torus manifold
$S^{2n} \#_{T^k} S^{2n}$. We denote the inclusion $T_p \subset S^{2n}$
by $\tau_p$ for $p=1, 2$.
\end{example}

\begin{example}
From the equivariant connected sum construction we have that the orbit space of
the torus manifold $S^{2n} \#_{T^k} S^{2n}$ is $Q^n \#_{x_1, x_2 } Q^n$. The
orbit spaces of $S^6$, $S^6 \#{T^1} S^6$, $S^6 \#_{T^2} S^6$ and $S^6 \#_{T^3} S^6$
are given in Figure \ref{tt} respectively. Note that the manifold with corners in
Figure \ref{tt} $(d)$ is homeomorphic to $\partial{Q^3} \times [0, 1]$ as manifold
with corners, so it is not contractible. 
\begin{figure}[ht]
        \centerline{
           \scalebox{0.72}{
            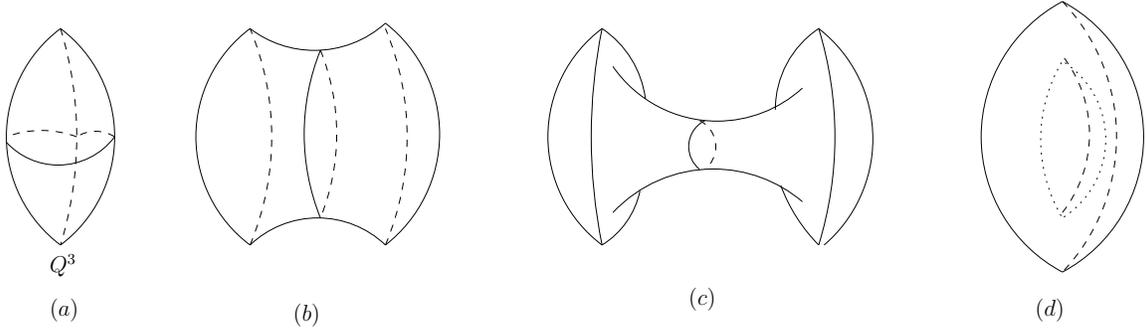
            }
          }
       \caption {Some connected some of manifolds with corners.}
       \label{tt}
      \end{figure}
\end{example}

\begin{lemma}\label{lem_conn_sum1}
Let $M$ be locally standard torus manifold with orbit space $P$.
If $P$ is a connected sum of simple polytopes, then $M$ is an equivariant connected
sum of toric manifolds. 
\end{lemma}
\begin{proof}
It is enough to prove when $P$ is the connected sum $Q \#_{x_F, y_G} R$
where $x_F$ and $y_G$ are interior points of the $k$-dimensional faces $F$
and $G$ of the simple polytope $Q$ and $R$ in $\RR^n$ respectively. Let $B$
and $C$ be open $n$-ball around $x_F$ and $y_G$ respectively in $\RR^n$ such that
conditions in \ref{con} are satisfied. So there are diffeomorphisms 
$$f : Q- B \to Q \#_{x_F, y_G} R ~ \mbox{and} ~ g : R-C \to Q \#_{x_F, y_G} R$$
onto its image. Let $\{P_1, \ldots, P_m\}$, $\{Q_1, \ldots, Q_k\}$ and $\{R_1, \ldots, R_l\}$
be the facets of $P, Q$ and $R$ respectively. Let $\xi : \{P_1, \ldots, P_m\} \to \ZZ^n$
be the  characteristic function of the torus manifold $M$. We define a function 
\begin{equation}
\eta_1: \{Q_1, \ldots, Q_k\} \to \ZZ^n~ \mbox{and} ~ \eta_2 : \{R_1, \ldots, R_l\} \to \ZZ^n
\end{equation}
by
\begin{equation}
\eta_1(F) = \xi(P_i) ~ \mbox{if} ~ f(F) = P_i \cap f(Q) ~\mbox{and} ~ \eta_2(F) = \xi(P_j) ~ \mbox{if}
~ g(F) = P_j \cap g(R)
\end{equation} respectively.
So $\eta_1$ and $\eta_2$ are characteristic function of $Q$ and $R$ respectively.
Let $M_1$ and $M_2$ be the toric manifolds over $Q$ and $R$ corresponding
to these characteristic functions respectively. Let $q \colon M_1 \to Q$ and
 $r \colon M_2 \to R$
be the orbit maps. Let $\hat{B} = q^{-1}(B)$ and $\hat{C} = r^{-1}(C)$.
Let $M_1 \#_{x_F, y_G} M_2$ be the space obtained by identifying the boundary
of $M_1 - \hat{B}$ and $M_2 - \hat{C}$ via an equivariant diffeomorphism.
Then $M$ is equivariantly homeomorphic to $M_1 \#_{x_F, y_G} M_2$.
\end{proof}

\begin{lemma}\label{lem_homeo_clasi}
Let $M$ and $N$ be $2n$-dimensional toric manifolds. Then $M \#_{T^k} N $
is homeomorphic to $M \# N \# (S^{2n} \#_{T^k} S^{2n}).$
\end{lemma}
\begin{proof}
Let $\pi_1 : M \to P_1$ and $\pi_2 : N \to P_2$ be two orbit maps. Let
$\iota_1 \colon T^k \to M$ and $\iota_2 \colon T^k \to N$ be embedding
of $k$-dimensional orbits. So $Im(\iota_1)$ and $Im(\iota_2)$ are interior
point of $k$-dimensional face $F$ and $G$ of $P_1$ and $P_2$ respectively.
Let $v_1$ and $v_2$ be vertices of $F$ and $G$ respectively. So $v_1$ and 
$v_2$ are vertices of $P_1$ and $P_2$ respectively.

We may assume the polytopes $P_1$ and $P_2$ are subset of $\RR^n$. Let $B_i$
be an open ball in $\RR^n$ with center $v_i$ such that $\bar{B_i \cap P_i} \cap F$
is empty for any face not containing the vertex $v_i$ for $i=1, 2$. Let
$U_i = B_i \cap P_i$ for $i=1, 2$. So $\pi_i^{-1}(U_i)$ is equivariantly
homeomorphic to $\CC^n$. Without loss of generality, we may consider that
$Im(\iota_i)$ belongs to $\pi_i^{-1}(U_i)$ for $i=1, 2$. So $M$ and $N$ are
equivariantly homeomorphic to $M \#_{v_1} S^{2n}$ and $N \#_{v_2} S^{2n}$ 
respectively such that $Im(\iota_1)$ and $Im(\iota_2)$ are $k$-dimensional
orbit in $S^{2n}$. 

Since $S^{2n}$ is a torus manifold, we can construct $S^{2n} \#_{T^k} S^{2n}$.
So $M \#_{T^k} N$ is equivariantly homeomorphic to
$M \#_{v_1} (S^{2n} \#_{T^k} S^{2n}) \#_{v_2} N$. Forgetting the
equivariantness of the connected sum  $M \#_{v_1} S^{2n}$ and $N \#_{v_2} S^{2n}$,
we get that $M \#_{T^k} N $ is homeomorphic to $M \# N \# (S^{2n} \#_{T^k} S^{2n}).$
\end{proof}

\begin{corollary}\label{cor:conn_tor_mfds}
Let $M_1, \ldots, M_n$ be $2n$-dimensional toric manifolds. Then
$$M_1 \#_{T^{k_1}} M_2 \#_{T^{k_2}} \cdots \#_{T^{k_{n-1}}} M_n $$ is homeomorphic
to $$M_1 \# \cdots \# M_n \# (S^{2n} \#_{T^{k_1}} S^{2n}) \# \cdots \# (S^{2n}
\#_{T^{k_{n-1}}} S^{2n}).$$ 
\end{corollary}
\begin{proof}
This follows from the successive application of Lemma \ref{lem_homeo_clasi}. 
\end{proof}

\section{Torus complement of even dimensional spheres}\label{tcev}
In this section we discuss the homotopy type of the complement of an orbit
of the torus action on even dimensional spheres. We show that if $T^k \subset S^{2n}$
is an orbit then $S^{2n} \setminus T^k$ is homotopic to wedge of some lower dimensional 
spheres. Let $[n] = \{1, \ldots, n\}$ for all positive integer $n$.
and $\mathcal{L}(n, 0)$ the set of faces of $Q^n$ including $Q^n$.
For $1 \leq k \leq n$, let $$\mathcal{L}(n, k)=
\mathcal{L}(n, 0) - \{F \in \mathcal{L}(n, 0) : F ~\mbox{contains the edges}~ \{E_k, \ldots, E_n\}\}$$ for $k=1, \ldots, n$.
Let $$V([n], k) := \bigcup_{F \in \mathcal{L}(n, k)} F ~\mbox{and} ~
U([n], k) := \pi^{-1}(V([n], k))~$$ for $k=0, 1, \ldots, n$. The following Proposition
true from the definition.

\begin{prop}
 $V([n], n) = V([n-1])$ and $U([n], n) = U([n-1])$ for all $n \geq 2$.
\end{prop}

Let $Q^{n-1}_k$ be the facet of $Q^n$ not containing the edge $E_k$ for
$k \in \{1, \ldots, n\}$. Define $V([n]-\{k-1\}) = Q^{n-1}_{k-1}$ for $k \in \{2, \ldots n\}$.
Then similarly as above we can define $V([n]-\{k-1\}, k)$ and $U([n]-\{k-1\}, k)$
for $k \in \{2, \ldots, n\}$.
\begin{prop}\label{l1}
For all $k$, $2 \leq k \leq n$ we have the following homotopy push-out diagram.
\begin{equation}
\begin{CD}
 V([n]-\{k-1\}, k) @>>> V([n], k)\\
@VVV @VVV\\
V([n]-\{k-1\}) @>>> V([n], k-1).
\end{CD}
\end{equation}
and
\begin{equation}
\begin{CD}
 U([n]-\{k-1\}, k) @>> f > U([n], k)\\
@VgVV @VVV\\
U([n] -\{k-1\}) @>>> U([n], k-1).
\end{CD}
\end{equation}
Moreover, $f$ and $g$ are null homotopic. 
\end{prop}
\begin{proof}
The maps in both the diagram in the Proposition \ref{l1} are natural. We'll show $f$ and $g$
are null homotopic. Since $U([n] -\{k-1\})$ is a sphere of dimension $2n-2$,
the map $g$ is null homotopic by dimension reason.

First we want to define a homotopy from $V([n]-\{k-1\}, k)$ to $E_{k-1}$ relative
to the vertices of $Q^n$. Let $z=(x_1, \ldots, x_{k-2}, 0, x_{k}, \ldots, x_n, x)
\in Q^n$. Then we have a continuous onto map 
$h \colon V([n]-\{k-1\}, k) \to E_{k-1}$ defined by 
$$(x_1, \ldots, x_{k-2}, 0, x_k, \ldots, x_n, x) \to (0, \ldots, 0, x_{k-1}, 0,
\ldots, 0, x).$$
Hence the map $H \colon V([n]-\{k-1\}, k) \times [0, 1] \to Q^n$ defined by
$$H(z, r) = \frac{(1-r) z + r h(z)}{||(1-r)z + r h(z)||}$$ gives a homotopy
from $V([n]-\{k-1\}, k)$ to $E_{k-1}$
relative to their vertices. Let $F_{k-1}=\{(x_1, \ldots, x_{k-2}, 0, x_{k},
\ldots, x_n, x) \in Q^n\}$.
Then the isotropy group, denote it by $T_{k-1}$, of $\pi^{-1}(F_{k-1})$ is
the $(k-1)$-th circle subgroup of $T^n$.
Let $$T^{n-1}_{k-1} := T^n/T_{k-1} = \{(t_1, \ldots, t_{k-2}, 1, t_k, \ldots, t_n)
\in T^n\}.$$ Note that
\begin{equation}\label{eq1}
(T^n \times V([n]-\{k-1\}, k))/\sim = (T^{n-1}_{k-1} \times V([n]-\{k-1\}, k))/\sim 
\end{equation}
 where the equivalence
relation $\sim$ is defined in \eqref{ereln}.
The homotopy $H$ induces a map  $$\widehat{H} \colon T^{n-1}_{k-1} \times V([n]-\{k-1\}, k)
\times [0, 1] \to T^n \times Q^n,$$ defined by
$$((t_1, \ldots, t_{k-2}, 1, t_k, \ldots, t_n), z, r) \to ((t_1, \ldots,
 t_{k-2}, 1, t_k, \ldots, t_n), H(z,r)).$$
Observe that the map $\widehat{H}$ induces a map, denoted by $\bar{H}$,
from $$(T^{n-1}_{k-1} \times V([n]-\{k-1\}, k))/\sim \times [0, 1]
\to (T^n \times Q^n)/\sim.$$ Now $$\bar{H}((t_1, \ldots, t_{k-2}, 1, t_k, \ldots,
 t_n), z, 1) = ((t_1, \ldots, t_{k-2}, 1, t_k, t_n), h(z))/\sim = h(z).$$
So $\bar{H}$ is a homotopy from $(T^n \times V([n]-\{k-1\}, k))/\sim$ to
$\pi^{-1}(E_{k-1}) \cong S^2$. Since the image of $h$ is homeomorphic to $[0, 1]$,
the map $f$ is null homotopic.
\end{proof}
\begin{lemma}
The space $U([n], k)$ is strong deformation retract of $S^{2n} - T^{n-k+1}$ for
$1 \leq k \leq n$ and $U([n]-\{k-1\}, k)$ is strong deformation retract of
$S^{2n-2} - T^{n-k+1}$ for $2 \leq k \leq n$. 
\end{lemma}
\begin{proof}
 Let $\bar{\pi} \colon Q^n \to \RR^{n+1}$ be the map defined by
$$\bar{\pi} ((x_1, \ldots, x_n, x)) = \left\{ \begin{array}{ll} (x_1^2, \ldots, x_n^2, \frac{x^3}{|x|}) & \mbox{if} ~ x \neq 0\\
(x_1^2, \ldots, x_n^2, 0) & \mbox{if} ~x=0.~
\end{array} \right.
$$
Let $P^n$ be the image of $\bar{\pi}$. So $\bar{\pi} \colon Q^n \to P^n$ is a homeomorphism.
Let $\xi \colon S^{2n} \to P^n$ be the composition of $\pi$ and $\bar{\pi}$.
Let $F_k=\{(0, \ldots, 0, y_k, \ldots, y_n, 0) \in P^n\}$ for $k \in \{1, \ldots, n\}$.
Let $z$ be an interior point of $F_k$. Then $\xi^{-1}(z) \cong T^{n-k+1}$. 
So $S^{2n} - T^{n-k+1} = S^{2n} - \xi^{-1}(z)$. Let $x \in P^n -\{z\}$ and
$\ell(z, x)$ be the line in $\RR^{n+1}$ passing through the points $\{z, x\}$.
From the definition of $V([n], k)$ we can see that $\ell(z, x)$ intersect
$\bar{\pi}(V([n], k))$ in a unique point, say $f_z(x)$. So this induces a continuous map
$f_z \colon P^n -\{z\} \to \bar{\pi}(V([n], k))$. We define a homotopy
$$H \colon P^n -\{z\} \times [0, 1] \to P^n -\{z\}$$ by 
$H(x, r) = (1-r)x+rf_z(x)$. Note that $H$ fixes $\bar{\pi}(V([n], k))$.
So $$\bar{H} := H \circ (\bar{\pi} \times Id)=
Q^n - \bar{\pi}^{-1}(z) \times [0,1] \to Q^n - \bar{\pi}^{-1}(z)$$ is a homotopy
from $Q^n - \bar{\pi}^{-1}(z)$ to $V([n], k)$.
If $\bar{\pi}^{-1}(x)$ belongs to the relative interior of a face $F$ of $Q^n$
then $\bar{\pi}^{-1}(f_z(x))$ belongs to the relative interior of a face of $F$. 
So $T_{F(x)}$ is a subgroup of $T_{F(H(x, r))}$ for every $r \in [0, 1]$.
Hence $\bar{H}$ induces a homotopy from $(T^n \times Q^n - \bar{\pi}^{-1}(z))/\sim$ to 
$(T^n \times V([n], k))/\sim$ defined by $$([t, x]^{\sim}, r) \to [t, \bar{H}(x, r)]^{\sim}.$$ 
This proves the first statement of Lemma. The second statement can be proved similarly.
\end{proof}

\begin{lemma}\label{l2}
Let $t=n+1-k$. Then $U([n], k) \simeq \vee_{i=1}^t \vee_{t \choose i} S^{2n-1-i}$.
\end{lemma}
\begin{proof}
Clearly Lemma hold for $n=1$ and $k=1$.
Now assume that Lemma holds for all $U(j, k)$ such that $j \leq n-1$ and for $j \geq k$.
So by this assumption $$U([n] - \{k-1\}, k) \simeq \vee_{i=1}^t \vee_{t \choose i} S^{2n-3-i},$$ where $t= n-k+1$.
For $k=n$, $t=1$ and $U([n], k) \simeq S^{2n} - S^1 \simeq S^{2n -2}$. So Lemma \ref{l2} holds.
Using the induction hypothesis and Proposition \ref{l1} we have the following
homotopy push-out.
\begin{equation}
\begin{CD}
\vee_{i=1}^t \vee_{t \choose i} S^{2n-3-i} @>> f > \vee_{i=1}^t \vee_{t \choose i} S^{2n-1-i}\\
@VgVV @VVV\\
S^{2n-2} @>>> U([n], k-1).
\end{CD}
\end{equation}
Since $f$ and $g$ are null homotopic, 
\begin{equation}
\begin{aligned}
 U([n], k-1) & \simeq \vee_{i=1}^t \vee_{t \choose i} S^{2n-1-i} ~\vee ~\vee_{i=1}^t \vee_{t \choose i} S^{2n -2-i} \vee S^{2n -2}\\
           & \simeq \vee_{i=1}^t \vee_{t \choose i} S^{2n-1-i}~ \vee~ \vee_{i=0}^t \vee_{t \choose i} S^{2n-2-i}\\
           & \simeq \vee_{i=1}^t \vee_{t+1 \choose i} S^{2n-1-i}, ~~ \mbox{since} ~~ {t \choose i} + {t \choose {i-1}} = {t+1 \choose i} ~~ \mbox{for} ~~ i \geq 1.
 \end{aligned}
\end{equation}
Which shows that Lemma \ref{l2} holds for $U([n], k-1)$.
\end{proof}

\section{Homology groups of certain torus manifolds}\label{sec:hom_gps}
In this section we compute the homology groups of torus manifolds which are equivariant connected sum of toric manifolds. First we compute the homology of $S^{2n} \#_{T^k} S^{2n}$ of Example \ref{tormfd} where $1 \leq k \leq n $. Other cases can be computed
using Corollary \ref{cor:conn_tor_mfds}. We adhere the notations of previous sections. 

\begin{prop}\label{prop:hom_conn_sum}
Let $n > 2$, $k'= \min\{k, n-3\}$, $q(k)=1$ if $k=n$, and $q(k)=0$ if $k=n-2, n-1$. 
Then homology groups of $S^{2n} \#_{T^k} S^{2n}$ are given by the following.
$$
H_i(S^{2n} \#_{T^k} S^{2n}) = \left\{ \begin{array}{ll} \ZZ^{k \choose j} & \mbox{if} ~ i = 2n-1 - j~ \mbox{and} ~ j \in \{1, \ldots, k'\}, \\
\ZZ^{q(k)} \oplus \ZZ^{k \choose n-2} & \mbox{if} ~ i=n+1, k =n-2, n-1, n \\
\ZZ^{k \choose n-1} \oplus \ZZ^{k \choose n-1} & \mbox{if} ~i=n, k > n-2,\\
\ZZ^{k \choose n-2} \oplus \ZZ^{q(k)} & \mbox{if} ~ i=n-1, k =n-2, n-1, n, \\
\ZZ^{k \choose j} & \mbox{if} ~ i = j+1~ \mbox{and} ~ j \in \{1, \ldots, k'\}, \\
 \ZZ & \mbox{if}~ i = 0, 2n, \\
 0 & \mbox{otherwise}.
\end{array} \right.
$$
\end{prop}

\begin{proof}
Let the orbit $T^k \hookrightarrow S^{2n}$ be $\pi^{-1}(x)$ such that $x$ belongs
to the relative interior $F^0$ of a $k$-dimensional face $F$ of $Q^n$. Without any
loss, we may assume the edges of $F$ are $E_{n-k+1}, \ldots, E_n$ of $Q^n$. 

Let $\alpha : [-1, 1] \to Q^n$ be a simple path in $Q^n$ defined by
$$
\alpha(t) = \left\{ \begin{array}{ll} 
 \frac{(1+t) x -t(0, \ldots, 0, -1)}{||(1+t)x -t(0, \ldots, 0, -1)||}  &
 \mbox{if}~ t \in [-1, 0], \\\\
 \frac{t x + (1-t)(0, \ldots, 0, 1)}{||tx + (1-t)(0, \ldots, 0, 1)||}  &
 \mbox{if } t \in [0, 1].
\end{array} \right.
$$
 joining the points $(0, \ldots, 0, -1)$ and $(0, \ldots, 0, 1)$, and passing through $x$.
 Note that $\alpha((0, 1)) \subseteq F^0$, see Figure \ref{tt2} for examples
when $n=3$. From the equivariant connected sum construction in subsection 
\ref{equi_cunn_sum}, we have the orbit space of $S^{2n} \#_{T^k} S^{2n}$ is
$Q^n \#_{x, x} Q^n$ and the corresponding orbit map given by $$\pi: S^{2n}
\#_{T^k} S^{2n} \to Q^n \#_{x, x} Q^n.$$ For simplicity we assume $Q_i=Q^n$
represents the $i$-th component in $Q^n \#_{x, x} Q^n$. We may perform the
connected sum such that $$\alpha([-1, 1]) \#_{x, x} \alpha([-1, 1]) = 
\alpha_1([-1, 1]) \sqcup \alpha([-1, 1]) \subset Q^n \#_{x, x} Q^n$$ such
that $\alpha_i : [-1, 1] \to Q^n \#_{x, x} Q^n$ satisfies $\alpha_i(0)=
\alpha(-1) \in Q_i$, $\alpha_i(1)=\alpha(1) \in Q_i$ and $\alpha_i(0, 1)
 \subseteq (F \#_{x, x} F)^0 \subset Q^n \#_{x, x} Q^n$ for $i=1, 2$. 
\begin{figure}[ht]
        \centerline{
           \scalebox{0.72}{
            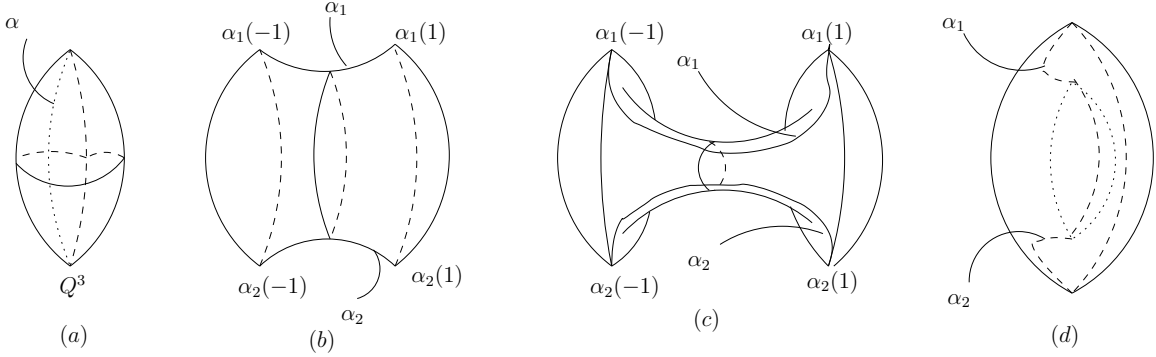
            }
          }
       \caption {The paths $\alpha$, $\alpha_1$ and $\alpha_2$ is some examples.}
       \label{tt2}
      \end{figure}

Recall from Section \ref{tcev} that there is a face preserving strong 
deformation from $Q^n - \{x\}$ onto $V([n], n-k+1) \subset Q^n$. We may
assume $V([n], n-k+1)$ is a subset of $Q_2$. Therefore from the connected
sum $Q^n \#_{x, x} Q^n$, we have that the complement of $$D := \alpha_2([-1, 1])
\cup V([n], n-k+1) \subset Q^n \#_{x, x} Q^n$$ is homeomorphic to
$\RR^n_{\geq 0}$ as manifold with corners. Let $B = Q^n \#_{x, x} Q^n \setminus
 D$ and $C = Q^n \#_{x, x} Q^n \setminus \{\alpha_1(-1)\}$.
Then following the arguments in Section \ref{tcev}, one can show that
there is a face preserving strong deformation from $B \setminus \{\alpha_1(-1)\}$ onto
a vertex cut at $\alpha_1(-1)$ in $Q^n \#_{x, x} Q^n$. Here the vertex cut
at $\alpha_1(-1)$ in $Q^n \#_{x, x} Q^n$ is homeomorphic to $\Delta^{n-1}$
as manifold with corners. So $\pi^{-1}(B \setminus \{\alpha_1(-1)\})$ is
homotopic to $S^{2n-1}$.

Also $D$ is deformation retract of $C$ in $Q^n \#_{x, x} Q^n$ by a face preserving
homotopy, and hence $\pi^{-1}(D)$ is a deformation retract of $\pi^{-1}(D_n)$. Therefore,
\begin{equation}
\pi^{-1}(C) \cap \pi^{-1}(D_n) \simeq S^{2n-1}.
\end{equation} 
We have $\alpha_i(-1, 1) \subset (F \#_{x, x} F)^0$. 
So $\pi^{-1}(\alpha_2([-1, 1]))$ is homeomorphic to
 the suspension $\Sigma T^k$,  since the image
$\alpha((0, 1))$ belongs to the interior of a $k$ dimensional face $F$. Then 
\begin{equation}
\pi^{-1}(D) = \Sigma T^k \vee U([n], n-k+1).
\end{equation}
By Lemma \ref{l2}, $$U([n], n-k+1) \simeq \vee_{i=1}^k \vee_{k \choose i} S^{2n-1-i}.$$ Note that dimension of $U([n], n-k+1)$ is $2n-2$. Then using Mayer-Vietoris sequence for the open subsets $\pi^{-1}(B) $ and $ \pi^{-1}(C)$ of $S^{2n} \#_{T^k} S^{2n}$, we get the lemma.
\end{proof}

Following the notation and proof of Proposition \ref{prop:hom_conn_sum}, we get the following when $n=2$. 
\begin{prop}
For $k \in \{1, 2\}$, the homology groups of $S^{4} \#_{T^k} S^{4}$ are given by the following.
$$
H_i(S^{4} \#_{T^k} S^{4}) = \left\{ \begin{array}{ll} 
\ZZ^{k-1}  & \mbox{if} ~ i=1, 3, k =1, 2,\\
\ZZ^{2} & \mbox{if} ~i=2,\\
 \ZZ & \mbox{if}~ i = 0, 4, \\
 0 & \mbox{otherwise}.
\end{array} \right.
$$
\end{prop}

\section{Cohomology ring of certain torus manifolds}\label{sec:cohom_ring}
In \cite{LS}, the authors discuss the rational homotopy type of the complement of a Poincare
embedding when the codimension is at least 3. We adhere the notations of previous Sections.
We know that $S^{2n}$ is a torus manifold with orbit space $Q^n$.
Let $$\tau_p \colon T^k \to S^{2n}$$ be embedding of $k$-dimensional
orbit for $p = 1, 2$ such that $Im(\tau_1) \cap Im(\tau_2)$ is empty.
We want to find a model for $S^{2n} \#_{T^k} S^{2n}$.
Note that connected sum depends on the identification of the boundary of 
normal neighborhood. Here we are considering the identification via identity
map (possibly with reverse orientation). In this section, we assume
$\mathbb{K}$ is a field of characteristic zero and $p \in \{1, 2\}$.

First, we compute a model for the embedding $\tau_p \colon T^k \to S^{2n}$.
 Let $N_p$ be a tubular neighborhood of $T^k$ in $S^{2n}$. Then the following
 diagram commutes with codimension of $N_p$ is zero. 
\[
\xymatrix{
T^k \ar[r]^{\tau_p} \ar[dr] & N_p \ar[d] \\
& S^{2n}
}
\]
So we have the inclusion $\partial{N}_p \to \bar{S^{2n} -N_p}$. This
induces a map $\phi \colon B_p \to C_p$ from a model $B_p$ of $\bar{S^{2n}-N_p}$
to a model $C_p$ of $\partial{N}_p$. Since $T^k$ is a $k$-dimensional
orbit, its tubular neighborhood $N_p$ is diffeomorphic to the normal
bundle of $T^k$ in $S^{2n}$ which we know is trivial. So $\partial{N_p}$
is homeomorphic to $T^k \times S^{2n-k-1}$. Now a model for $T^k$ is
$\Lambda(a_1, \ldots, a_k : da_j=0)$. A model for $S^{2n-k-1}$
is $\Lambda(b: db=0)$, where $|b|=2n-k-1$. So we get

\begin{prop} A model for $\partial{N_p}$ is
\begin{equation}
C = \Lambda(a_1, \ldots, a_k : da_j =0) \tensor \Lambda(b : db=0) \cong
\Lambda(a_1, \ldots, a_k, b : d=0).
\end{equation}
\end{prop}

\begin{prop}
A model for $\bar{S^{2n} -N_p}$ is
\begin{equation}
 B= \QQ 1 \oplus \QQ c \oplus b^{\prime} \cdot \Lambda(a_1, \ldots, a_k)
\end{equation}
where $|a_i|=1, |c|=2n, |b^{\prime}|=2n-k-1$ and $d(b^{\prime} a_1 \ldots a_k) =c$,
for all other basis elements $d=0$.
\end{prop}
\begin{proof}
By \cite[Theorem 1.6]{LS} we get that as a $\Lambda (c)$-module the model
of $\bar{S^{2n} - N_p}$ in the proposition is correct. Now by results from
Section \ref{tcev}, the space $\bar{S^{2n}-N_p}$ is a wedge of spheres
and so $(b^{\prime})^2=0$. All other products in $B$ are zero for dimension
reason. Therefore the result follows. 
\end{proof}

\begin{lemma}
Let $A$ be a CDGA with trivial product structure and $\Lambda V \to A$ be
a model of $A$. Let $f, g \colon \Lambda V \to B$ be CDGA maps where $B$
has zero differential such that $H^i(f)$ and $H^i(g)$ are surjective for
$2con(A) < i$. Then there exists $\phi \colon \Lambda V \to \Lambda V$
such that $ f \phi \simeq g$.
\end{lemma}
\begin{proof}
Since product in $A$ is trivial, if necessary, one may first consider an
automorphism of $A$ so that $H_i(f) = H_i(g)$. 
Let $\{a_\alpha\}$ be a graded basis of $\oplus H^{\ast}(A)$ for $\ast \leq \dim{B}$.
Then there is a minimal model $\phi \colon \Lambda V \to A$ with $V = \<a_{\alpha}, b_j\>$
 such that $d{a_{\alpha}} = 0$, $|b_j|>2con(A)$ and $db_j\in \Lambda^{\geq 2} V$.
Suppose we have already defined $\phi_n \colon \Lambda V^{\leq n} \to \Lambda V$ 
such that $f \phi_{n} = g|_{\Lambda V^{\leq n}}$.

Let $\phi^{\prime}_{n+1}$ be any extension of $\phi_n$ (this always exists 
since all cycles in $\Lambda^{\geq 2} V$ are boundaries).

Let $\{v_i\}$ be a basis of $V^{n+1}$. Then $x_i :=(f \phi_{n+1}^{\prime} -g)(v_i) \in H^{n+1}B$. Therefore, $$x_i = f(\Sigma_{j} c_j a_{\alpha_j}).$$
Define $\phi_{n+1}(v_i) = \phi_{n+1}^{\prime}(v_i) - \Sigma_{j} c_j a_{\alpha_j}$.
Then $f \phi_{n+1} = g|_{\Lambda V^{\leq n+1}}$. Continuing in this way we get the required map $\phi$. 
\end{proof}

\begin{corollary}\label{a26}
Let $f \colon X \to Y$ be a map such that $H^n(f; \QQ)$ is a surjection for $n>2conY$.
Suppose that $Y$ is rationally a wedge of spheres, $A$ is a model of $Y$, $B$ is
a model of $X$ and  $\phi \colon A \to B$ is a CDGA map such that $H(\phi) = H^{\ast}(f)$.
Then $\phi$ is a model of $f$.
\end{corollary}

\begin{prop}\label{a24}\label{prop:model_inclution}
 A model for the inclusion $\iota \colon \partial{N_p} \to \bar{S^{2n}-N_p}$ is
 $\phi \colon B \to C$ determined by $1 \to 1$, $c \to 0$ and
 $b^{\prime}\alpha \to b \alpha$ for $\alpha \in \Lambda(a_1, \ldots, a_k)$.
\end{prop}
\begin{proof}
First we know that $\phi$ is a CDGA map, and also by \cite[Theorem 1.6]{LS}
$\phi$ is a model of the map $\iota$ as $\Lambda(c)$-modules. 
Thus it induces the correct map on cohomology. Then Corrollary \ref{a26} 
implies that it is a model for the inclusion $\iota$. 
\end{proof}


%
%


Now, we compute some model for $S^{2n} \#_{T^k} S^{2n}$.
Consider the subvector space $\mathcal{I} \subset \Lambda(a_i, da_i)$ with basis
$$\{a_{i_1} \cdots a_{i_j} da_{s_1} \cdots da_{s_k} | k \geq 2 ~\mbox{and if} ~ k=1 ~\exists ~t \in \{1, \ldots, j\} ~\mbox{with} ~ i_t \leq s_1\}.$$
\begin{lemma}\label{a28}
Then $\mathcal{I}$ is an acyclic differential ideal.
\end{lemma}
\begin{proof}
It is clear that $\mathcal{I}$ is a differential ideal. Since $\Lambda(a_i, da_i)$ is acyclic,
it is equivalent to show $E^{\prime} := \Lambda(a_i, da_i)/\mathcal{I}$ is acyclic. $E^{\prime}$ has a basis for representative
$$a_{i_1}\cdots a_{i_j} ~ \mbox{and} ~ (da_{i_1})a_{i_2} \cdots a_{i_j}~ \mbox{with} ~ i_1 \leq \cdots \leq i_j.$$
Now $d(a_{i_1} \cdots a_{i_j}) = (da_{i_1}) a_{i_2} \cdots a_{i_j}$ in $E^{\prime}$. So $E^{\prime}$ is clearly acyclic.
\end{proof}

Note that the map $\phi$ in Proposition \ref{prop:model_inclution} is not surjective.
So we replace $\phi \colon B \to C$ with 
$$\phi^{\prime} \colon E' \tensor B  \twoheadrightarrow \Lambda(a_i, b) = C.$$ determined by
$$a_i \to a_i, ~ da_i \to 0, ~ c \to 0, ~ \mbox{and} ~ \phi'|_{B} = \phi.$$
Let $B^{\prime} := E' \tensor B$. Then $\phi^{\prime} \colon B' \to C$
is a surjective model for $\iota \colon \partial{N_p} \to \bar{S^{2n} \setminus N_p}$ so that
$$a_i \to a_i, da_i \to 0 ~ \mbox{and} ~ b^{\prime} a_{i_1} \cdots a_{i_\ell} \to b a_{i_1} \cdots a_{i_\ell}$$ for $\{i_1, \ldots, i_\ell\} \subseteq \{1, \ldots, k\}$.
Then the homotopy pull back of
\begin{equation}
 \begin{CD}
 @. B^{\prime}\\
@. @V \phi^{\prime} VV\\
B @> \phi >> C
 \end{CD}
\end{equation}
is $\mathcal{D} = \ker{(\phi^{\prime} - \phi)}$, where
 $\phi^{\prime} -\phi \colon B^{\prime} \oplus_{\QQ} B \to C$.
Let $a_I = a_{i_1} \cdots a_{i_\ell}$ and $\alpha_I = \alpha_{i_1} \cdots \alpha_{i_\ell}$
for $I = \{i_1, \ldots, i_\ell\} \subseteq \{1, \ldots, k\}$ with ordering on it.
Observe that $\mathcal{D}$ is generated by the elements of the form
as graded module over $\QQ$
\begin{equation}\label{eq:gen_D}
\{(da_{I'}, 0), (c, 0), (0, c), (a_{I} \tensor b',
 (-1)^{|I|}b' a_{I}), (1 \tensor b'a_I, b' a_I), 
(a_I \tensor b'~ + ~\mbox{sign}(\rho) a_{J} \tensor b^{\prime}a_{J'}, 0)\}
\end{equation}
where $\rho$ is the permutation from $I$ to $J \sqcup J'$, and
$I, I' J, J' \subseteq \{1, \ldots, k\}$ with $I' \neq \emptyset$.

Therefore from well-known facts, we get the following.

\begin{prop}
$\mathcal{D}$ is a model for $S^{2n} \#_{T^k} S^{2n}$.
\end{prop}



Let $I \subseteq \{1, \ldots, k\}$ and $I^c = \{1, \ldots, k\} -I$. Instead of 
computing $\mathcal{D}$ carefully, we construct
a map $\xi \colon A \to \mathcal{D}$ that is a quasi-isomorphism.
 We define $\xi$ by 
\begin{equation}\label{eq:defn_xi}
s^{-1}\alpha_{i_1} \cdots \alpha_{i_\ell} \to (da_{i_1}a_{i_2} \cdots a_{i_\ell}, 0),~
1 \to 1, s^{-2n+1} \alpha_I^{\#} \to (b^{\prime} a_{I^c}, b^{\prime}a_{I^c}), ~~
\mu_{2n} \to (0, c),
\end{equation}
where $I \neq \emptyset$. 
This may not define an algebra map. So we need to replace the image of $\xi$ by
a quasi-isomorphic CDGA, possibly by a quotient 
$\mathcal{D}/\mathcal{J}$ for some acyclic ideal $\mathcal{J}$.


Let $\mathcal{J}$ be the ideal of $\mathcal{D}$ generated by the following.
\begin{equation}\label{eq:elements_of_J}
\begin{array}{l}
 \{\alpha \in \mathcal{D}^{2n} : \alpha =(c, c)~ \mbox{or} ~ 
 (da_I \tensor b^{\prime}a_{I^c} - c, 0)\} \cup \\
\{\alpha \in \mathcal{D}^{2n-1} ~|~ \alpha =
 (a_{\{1, \ldots, k\}} \tensor b', (-1)^kb'a_{\{1, \ldots k\}}) ~ \mbox{or} ~
 (a_I \tensor b'a_{I^c} + ~\mbox{sign}(\rho) b^{\prime} a_{\{1, \ldots, k\}}, 0)\} \\
 \cup \{(da_I \tensor b'a_J, 0), (a_I \tensor b'a_J, (-1)^{|I|} b'a_I a_J) ~:~ \mbox{unless} ~ I= J^c, I \neq \emptyset \}~ \cup \\
 \{(da_{I'} \tensor c , 0), (da_{I'} \tensor b'a_{\{1, \ldots, k\}} , 0) ~: ~ I' \neq \emptyset\}
\end{array}
\end{equation}
where $\rho$ is a permutation of $I^c \cup I = \{1, \ldots, k\}$.


\begin{prop}\label{prop:basis_of_J}
Let $\mathcal{J}'$ be the submodule of $\mathcal{D}$ spanned by the elements
in \eqref{eq:elements_of_J}. Then $\mathcal{J} = \mathcal{J}'$. In particular,
the elements in \eqref{eq:elements_of_J} is a basis of $\mathcal{J}$.
\end{prop}
\begin{proof}
We need to show that $\mathcal{J} \subseteq \mathcal{J}'$, that is equivalent to
show $x \cdot y \in \mathcal{J}'$ whenever $x$ and $y$ belong to
\eqref{eq:gen_D} and \eqref{eq:elements_of_J} respectively. To see the following,
recall the algebra structure on $B$ and $B'$.
\begin{enumerate}
\item 
\begin{enumerate}
\item $(da_{I'}, 0) \cdot (c, c) = (da_{I'} \tensor c , 0)$.
\item $(c, 0) \cdot (c, c) = (c^2 , 0) = (0,0).$
\item $(0, c) \cdot (c, c) = (0, 0).$
\item $(a_I \tensor b', (-1)^{|I|} b' a_I) \cdot (c, c) = (0, 0)$.  
\item $(1 \tensor b'a_I,  b' a_I) \cdot (c, c) = (1 \tensor b' a_I c, b'a_Ic) = (0,0)$.  
\item $(a_I\tensor b' + \mbox{sign}(\rho) a_J \tensor b' a_{J'}, 0) \cdot (c, c)
= (a_I\tensor b'c + \mbox{sign}(\rho) a_J \tensor b' a_{J'} c, 0) = (0, 0).$\\
\end{enumerate}

\item 
\begin{enumerate}
\item $(da_{I'}, 0) \cdot ( da_{K} \tensor b' a_{K^c} -c, 0) = (da_{I'} \tensor c, 0)$.
\item $(c, 0) \cdot ( da_{K} \tensor b' a_{K^c} -c, 0) = ( 0, 0).$
\item $(0, c) \cdot ( da_{K} \tensor b' a_{K^c} -c, 0) = (0, 0).$
\item $(a_I \tensor b', (-1)^{|I|} b' a_I) \cdot ( da_{K} \tensor b' a_{K^c} -c, 0)
 = (a_Ida_K \tensor b'^2 a_{K^c} - a_I \tensor b'c, 0) = (0,0)$. 
 \item $(1 \tensor b'a_I,  b' a_I) \cdot  ( da_{K} \tensor b' a_{K^c} -c, 0)= (0,0)$.
\item $(a_I\tensor b' + \mbox{sign}(\rho) a_J \tensor b' a_{J'}, 0) \cdot
( da_{K} \tensor b' a_{K^c} -c, 0)  = (0, 0)$.\\
\end{enumerate}

\item 
\begin{enumerate}
\item $(da_{I'}, 0) \cdot (a_{\{1, \ldots, k\}} \tensor b', (-1)^k b' a_{\{1, \ldots, k\}}) 
= (da_{I'} a_{\{1, \ldots, k\}} \tensor b', 0) = (0, 0)$.
\item $(c, 0) \cdot (a_{\{1, \ldots, k\}} \tensor b', (-1)^k b' a_{\{1, \ldots, k\}})
 = ( a_{\{1, \ldots, k\}} \tensor c b', 0) = ( 0, 0).$
\item $(0, c) \cdot (a_{\{1, \ldots, k\}} \tensor b', (-1)^k b' a_{\{1, \ldots, k\}})
 = (0,0).$
\item $(a_I \tensor b', (-1)^{|I|} b' a_I) \cdot (a_{\{1, \ldots, k\}} \tensor b',
 (-1)^k b' a_{\{1, \ldots, k\}}) = (0, 0)$.  
 \item $(1 \tensor b'a_I,  b' a_I) \cdot (a_{\{1, \ldots, k\}} \tensor b', 
 (-1)^k b' a_{\{1, \ldots, k\}}) = ( 0, 0)$.
\item $(a_I\tensor b' + \mbox{sign}(\rho) a_J \tensor b' a_{J'}, 0) \cdot
 (a_{\{1, \ldots, k\}} \tensor b', (-1)^k b' a_{\{1, \ldots, k\}}) = (0,0).$\\
\end{enumerate}

\item 
\begin{enumerate}
\item $(da_{I'}, 0) \cdot (a_K \tensor b' a_{K^c} + \mbox{sign}(\rho) b' a_{\{1,
\ldots, k\}}, 0)\\
= (da_{I'}a_K \tensor b' a_{K^c} + \mbox{sign}(\rho) da_{I'} \tensor b' a_{\{1,
\ldots, k\}}, 0)\\
 = (da_L \tensor b' a_{K^c} + \mbox{sign}(\rho) da_{I'} \tensor b' a_{\{1, 
 \ldots, k\}}, 0) \in \mathcal{J}' $.
\item $(c, 0) \cdot (a_K \tensor b' a_{K^c} + \mbox{sign}(\rho) b' a_{\{1,
\ldots, k\}}, 0) = ( 0, 0) .$
\item $(0, c) \cdot (a_K \tensor b' a_{K^c} + \mbox{sign}(\rho) b' a_{\{1,
\ldots, k\}}, 0) = ( 0, 0).$
\item $(a_I \tensor b', (-1)^{|I|} b' a_I) \cdot (a_K \tensor b' a_{K^c} + 
\mbox{sign}(\rho) b' a_{\{1, \ldots, k\}}, 0) = (0, 0)$.
 \item $(1 \tensor b'a_I,  b' a_I) \cdot (a_K \tensor b' a_{K^c} +
  \mbox{sign}(\rho) b' a_{\{1, \ldots, k\}}, 0) = (0,0)$.  
\item $(a_I\tensor b' + \mbox{sign}(\rho) a_J \tensor b' a_{J'}, 0) \cdot
 (a_K \tensor b' a_{K^c} + \mbox{sign}(\rho) b' a_{\{1, \ldots, k\}}, 0) = (0, 0).$\\
\end{enumerate}

\item In this case $K \neq L^c$.
\begin{enumerate}
\item $(da_{I'}, 0) \cdot (da_K \tensor b'a_L, 0) = (0, 0), \mbox{as} ~ I', K \neq \emptyset.$
\item $(c, 0) \cdot (da_K \tensor b'a_L, 0) = ( 0, 0).$
\item $(0, c) \cdot (da_K \tensor b'a_L, 0) = ( 0, 0).$
\item $(a_I \tensor b', (-1)^{|I|} b' a_I) \cdot (da_K \tensor b'a_L, 0) = (0, 0)$.
 \item $(1 \tensor b'a_I,  b' a_I) \cdot (da_K \tensor b'a_L, 0) =( 0, 0)$.  
\item $(a_I\tensor b' + \mbox{sign}(\rho) a_J \tensor b' a_{J'}, 0) \cdot
 (da_K \tensor b'a_L, 0) = (0, 0).$\\
\end{enumerate}

\item In this case $K \neq L^c$.
\begin{enumerate}
\item $(da_{I'}, 0) \cdot (a_K \tensor b'a_L, (-1)^{|K|}b'a_Ka_L) =
(da_{K'} \tensor b'a_L, 0) \in \mathcal{J}'$, where $K' =I' \cup K$.
\item $(c, 0) \cdot (a_K \tensor b'a_L, (-1)^{|K|}b'a_Ka_L) = ( 0, 0).$
\item $(0, c) \cdot (a_K \tensor b'a_L, (-1)^{|K|}b'a_Ka_L) = (0, 0).$
\item $(a_I \tensor b', (-1)^{|I|} b' a_I) \cdot (a_K \tensor b'a_L, (-1)^{|K|}b'a_Ka_L)
=(0, 0) $.  
 \item $(1 \tensor b'a_I,  b' a_I) \cdot (a_K \tensor b'a_L, (-1)^{|K|}b'a_Ka_L) =(0, 0)$.
\item $(a_I\tensor b' + \mbox{sign}(\rho) a_J \tensor b' a_{J'}, 0) \cdot
 (a_K \tensor b'a_L, (-1)^{|K|}b'a_Ka_L) = (0, 0)$.\\
\end{enumerate}

\item 
\begin{enumerate}
\item $(da_{I'}, 0) \cdot (da_{J'} \tensor c , 0) = (0 , 0)$.
\item $(c, 0) \cdot (da_{J'} \tensor c , 0) = (da_{J'} \tensor c^2 , 0) = (0,0).$
\item $(0, c) \cdot (da_{J'} \tensor c , 0) = (0, 0).$
\item $(a_I \tensor b', (-1)^{|I|} b' a_I) \cdot (da_{J'} \tensor c , 0) = (0, 0)$.  
\item $(1 \tensor b'a_I,  b' a_I) \cdot (da_{J'} \tensor c , 0) (0,0)$.  
\item $(a_I\tensor b' + \mbox{sign}(\rho) a_J \tensor b' a_{J'}, 0) \cdot 
(da_{J'} \tensor c , 0) = (0, 0).$\\
\end{enumerate}

\item 
\begin{enumerate}
\item $(da_{I'}, 0) \cdot (da_{J'} \tensor b'a_{\{1, \ldots, k\}} , 0)
 = (0 , 0)$.
\item $(c, 0) \cdot (da_{J'} \tensor b'a_{\{1, \ldots, k\}} , 0) = (0,0).$
\item $(0, c) \cdot (da_{J'} \tensor b'a_{\{1, \ldots, k\}} , 0) = (0, 0).$
\item $(a_I \tensor b', (-1)^{|I|} b' a_I) \cdot (da_{J'} \tensor b'a_{\{1, \ldots, k\}} , 0) = (0, 0)$.  
\item $(1 \tensor b'a_I,  b' a_I) \cdot (da_{J'} \tensor b'a_{\{1, \ldots, k\}} , 0) = (0,0)$.  
\item $(a_I \tensor b' + \mbox{sign}(\rho) a_J \tensor b' a_{J'}, 0) \cdot 
(da_{J'} \tensor b'a_{\{1, \ldots, k\}} , 0) = (0, 0).$
\end{enumerate}

\end{enumerate} 
\end{proof}


\begin{prop}\label{30a}
The quotient map $\pi : \mathcal{D} \to \mathcal{D}/\mathcal{J}$ is a CDGA map and quasi-isomorphism.
\end{prop}
\begin{proof}
By Proposition \ref{prop:basis_of_J}, the elements of \eqref{eq:elements_of_J} 
form a graded basis of $\mathcal{J}$. Note that 
\begin{enumerate}
\item $d(b'a_{\{1, \ldots, k\}}, b' a_{\{1, \ldots, k\}}) =(c, c)$. 
\item $d(a_I \tensor b'a_{I^c} - b'a_{\{1, \ldots, k\}}, 0) = (da_I \tensor b' a_{I^c}, 0)$.
\item $d(a_I \tensor b' a_J, (-1)^{|I|} b' a_I a_J) = (da_I \tensor b' a_J, 0)$. 
\item $d(da_{J'} \tensor b'a_{\{1, \ldots, k\}}, 0) = (da_{J'} \tensor c , 0)$.
\end{enumerate}
Therefore we can conclude that $\mathcal{J}$ is an acyclic differential ideal.
\end{proof}

Let $\bar{B} := E' \oplus_{\QQ} B$. Then the natural map $E' \oplus_{\QQ} B \to
E' \tensor B$ is a differential graded module (DGM) chain map which induces an
isomorphism in cohomology. Also  $\bar{\phi} \colon \bar{B} \to C$
is a surjective DGM map so that
$$a_i \to a_i, da_i \to 0 ~ \mbox{and} ~ b^{\prime} a_{i_1} \cdots a_{i_\ell} \to
 b a_{i_1} \cdots a_{i_\ell}$$ where $\{i_1, \ldots, i_\ell\} \subseteq 
 \{1, \ldots, k\}$. Then the homotopy pull back of
\begin{equation}
 \begin{CD}
  @. \bar{B}\\
@. @V \bar{\phi} VV\\
B @> \phi >> C
 \end{CD}
\end{equation}
is $\bar{\mathcal{D}} = \ker{(\bar{\phi} - \phi)}$, where
 $\bar{\phi} - \phi \colon \bar{B} \oplus_{\QQ} B \to C$.
Note that $\bar{\mathcal{D}}$ is generated as graded module by the elements of the form
\begin{equation}\label{eq:basis_D_bar}
\{(da_{I}, 0), ~~(c, 0), ~~(0, c), ~~ (b'a_{i_1} ~~\cdots ~~a_{i_\ell} ,
 b' a_{i_1} \cdots a_{i_\ell})\},
\end{equation}
where $I, \{i_1, \ldots, i_\ell\} \subseteq \{1, \ldots, k\}$.
Therefore, the natural inclusion $\iota_D \colon \bar{\mathcal{D}} \to \mathcal{D}$ is a
differential graded module (DGM) map and quasi-isomorphism.
\begin{prop}
$\bar{\mathcal{D}}$ is a model for $S^{2n} \#_{T^k} S^{2n}$ and the set \eqref{eq:basis_D_bar} is a basis of $\bar{\mathcal{D}}$ over $\QQ$.
\end{prop}
\begin{proof}
The first statement follows from the well-known facts associated to homotopy pull-back 
of a diagram. For the rest, let $1 \leq i \leq 2n-1$ and $\bar{\mathcal{D}}^i$
consists of $i$-th degree elements. Then dimension of $\bar{\mathcal{D}}^i$ and
 $H_{2n-i}(S^{2n} \#_{T^k} S^{2n})$ are same, and when $i=2n$, $\bar{\mathcal{D}}^i$ is
 generated by $\{(c, 0), (0, c)\}$. 
\end{proof}

 We remark that the map $\xi$
factors through $\iota_D$. Let $\eta \colon A \to \bar{\mathcal{D}}$ be the map such
that $\xi = \iota_D \circ \eta$. 
\begin{prop}
The map $\eta \colon A \to \bar{\mathcal{D}}$ is injective DGM and chain map.
\end{prop}
\begin{proof}
Recall that differential on $A$ is zero and
$\{\alpha_{I} ~:~ I \subseteq \{1, \ldots, k\}\}$
is a graded basis for $\widetilde{H}^*T^k$. Since $\alpha_I$ and
$\alpha_I^{\#}$ are in complementary dimension, then
$$\{\mu_{2n}, ~s^{-1}\alpha_{I}, ~s^{-2n+1}\alpha_I^{\#} ~:~ I \subseteq
 \{1, \ldots, k\}, I \neq \emptyset\}$$ 
is a graded basis for $A$. Note that the map $\eta$ is same as $\xi$ as
in \eqref{eq:defn_xi}. So $\eta$ is a graded module map and injectivity
follows from the fact their image is a part of a basis as module over $\QQ$.  

It is a chain map since differential on $A$ is zero and differential on 
the image of each basis elements of $A$ is zero, that is $d(da_{I}, 0) = (0,0),$
$d(0, c)=(0,0),$ $ d(b'a_{i_1} \cdots a_{i_\ell}, b' a_{i_1} \cdots a_{i_\ell})=(0,0)$
for $\{i_1, \ldots, i_\ell\} \neq \{1, \ldots, k\}$. 
\end{proof}

We have the following commutative diagram of DGM. 
\begin{equation}
 \begin{CD}
 A @>{\eta}>> \bar{\mathcal{D}}\\
@V{\pi \xi}VV @V \iota_D VV\\
\mathcal{D}/\mathcal{J} @< \phi<< \mathcal{D}.
 \end{CD}
\end{equation}

\begin{prop}\label{30b}
 The map $\pi \xi : A \to \mathcal{D}/ J $ is a CDGA map and quasi-isomorphism.
\end{prop}
\begin{proof}
First we show that $\pi \xi$ is a quasi-isomorphism. So it is sufficient to 
show $\eta$ is quasi-isomorphism which is equivalent to show that 
$\bar{\mathcal{D}}/A$ is acyclic. The later is true since the ring
$\bar{\mathcal{D}}/A$ can be generated by $\{(c, c),
(b'a_{\{1, \ldots, k\}} , b'a_{\{1, \ldots, k\}})\}$ as a graded module.  

To show the other part, we need to show $\pi \xi$ is a chain map and $\pi \xi$ 
is an algebra map. Recall that $\{s^{-1} a_I, s^{-2n+1}a_{J}^{\#}, \mu_{2n} ~:~
I, J \subseteq  \{1, \ldots, k\}, I \neq \emptyset \neq J\}\}$ is a graded
basis of $A$. We have the following:
\begin{enumerate}
\item $d(\pi \xi) (s^{-1}(\alpha_{i_1} \cdots \alpha_{i_\ell})) = d (da_{i_1} a_{i_2}
 \cdots a_{i_\ell}) +\mathcal{J} = 0 + \mathcal{J}.$

\item $d (\pi \xi)(s^{-2n+1}\alpha_I^{\#}) = d (b'a_{I^c} , b'a_{I^c}) + \mathcal{J}
= 0 + \mathcal{J}$, since $d (b'a_{I^c} , b'a_{I^c}) = (0,0)$ if $I \neq \emptyset$ and 
$d (b'a_{I^c} , b'a_{I^c}) = (c, c)$ if $I = \emptyset$. 

\item $d (\pi \xi) (\mu_{2n}) = d(0, c) + \mathcal{J} = (0, 0) + \mathcal{J}$. 
\end{enumerate}

Thus $\pi \xi$ is a chain map. Now to show $\pi \xi$ is an algebra map.
We check this on the above graded basis of $A$. Recall the multiplication on $A$
from \eqref{eq:mult_in_A}. Let $I, J \subseteq \{1, \ldots, k\}$
and $I =\{i_1, \ldots, i_s\}, ~J=\{j_1, \ldots, j_t\}$.

\begin{enumerate}
\item $\pi \xi (s^{-1}\alpha_I \cdot s^{-1} \alpha_J) = 0 = 
 \pi \xi (s^{-2n+1}\alpha_I^{\#} \cdot s^{-2n+1} \alpha_J^{\#})$. It follows from the 
 multiplication in $A$.

\item \[
\begin{aligned}
 \pi \xi (s^{-1}\alpha_I) \cdot \pi \xi (s^{-1} \alpha_J) & = 
 \pi(da_{i_1} a_{i_2} \cdots a_{i_s}, 0) \cdot \pi(da_{j_1} a_{j_2} \cdots a_{j_t}, 0)\\
           & = \pi(da_{i_1} a_{i_2} \cdots a_{i_s} da_{j_1} a_{j_2} \cdots a_{j_t}, 0)\\
           & = 0 \in \mathcal{D}/\mathcal{J}.
 \end{aligned}
\]

\item 
\[
\begin{aligned}
 \pi \xi (s^{-1}\alpha_I) \cdot \pi \xi (s^{-2n+1} \alpha_J^{\#}) & = 
 \pi(da_{i_1} a_{i_2} \cdots a_{i_s}, 0) \cdot \pi(b'a_{J^c}, b'a_{J^c})\\
           & = \pi(da_{i_1} a_{i_2} \cdots a_{i_s} \tensor b'a_{J^c}, 0)\\
           & = \left\{ \begin{array}{ll}  (c, 0) + \mathcal{J} & \mbox{if}~ I = J \\
 (0,0) + \mathcal{J} & \mbox{otherwise in} ~ \mathcal{D}/\mathcal{J}.
\end{array} \right.
 \end{aligned}
\]

Also, $\pi \xi (s^{-1}\alpha_I \cdot s^{-2n+1} \alpha_J^{\#}) =
\pi (\delta_{IJ} \mu_{2n}) = \delta_{IJ}(0, c) = -\delta_{IJ} (c, 0)
\in \mathcal{D}/\mathcal{J}$.

\item
\[
\begin{aligned}
 \pi \xi (s^{-2n+1}\alpha_I^{\#}) \cdot \pi \xi (s^{-2n+1} \alpha_J^{\#}) & = 
 \pi(b'a_{I^c}, b'a_{I^c}) \cdot \pi(b'a_{J^c}, b'a_{J^c})\\
          & =  ((-1)^{|I^c|} b'^2 a_{I^c} a_{J^c}, (-1)^{|I^c|}b'^2 a_{I^c} a_{J^c}) \\            
 & = 0 + \mathcal{J}.
 \end{aligned}
\]

\item 
$\pi \xi (1 \cdot x) = \pi \xi (x) = \pi \xi(1) \cdot \pi \xi (x)$
for all basis element $x$ of $A$.

\item 
\begin{enumerate}
\item 
\[
\begin{aligned}
 \pi \xi (\mu_{2n}) \cdot \pi \xi (s^{-1} \alpha_I) & = 
 \pi(0, c) \cdot \pi(da_{i_1} \cdots a_{i_s}, 0)\\            
 & = 0 + \mathcal{J}\\
 & = \pi \xi(\mu_{2n} \cdot s^{-1} \alpha_I).
 \end{aligned}
\]

\item 
\[
\begin{aligned}
 \pi \xi (\mu_{2n}) \cdot \pi \xi (s^{-2n+1} \alpha_J^{\#}) & = 
 \pi(0, c) \cdot \pi(b'a_{J^c}, b'a_{J^c})\\            
 & = \pi (0, cb'a_{J^c}) \\
 & = \pi (0, 0)\\
 & = \pi \xi(\mu_{2n} \cdot s^{-1} \alpha_I).
 \end{aligned}
\]

\item
\[
\begin{aligned}
 \pi \xi (\mu_{2n}) \cdot \pi \xi (\mu_{2n}) & = 
 \pi(0, c) \cdot \pi(0, c)\\     
 & = (0, 0) + \mathcal{J}\\
 & =  \pi \xi(\mu_{2n} \cdot \mu_{2n}).
 \end{aligned}
\] 
\end{enumerate}

\end{enumerate}
\end{proof}


In the rest of this section, we compute the cohomology ring of connected toric
manifolds.

\begin{theorem}\label{thm:cohom_conn_sum}
Let $M$ and $N$ be $2n$-dimensional toric manifolds. Then the cohomology ring of $M \#_{T^k} N$ is given by the ring $R(M,N, T^k)$ of \eqref{eq:cohom_ring_conn_sum}.
\end{theorem}

\begin{proof}
By Lemma \ref{lem_homeo_clasi} $M \#_{T^k} N$ is homeomorphic to $M \# N \# (S^{2n} 
\#_{T^k} S^{2n})$. Now applying the following process 3 times one can prove the theorem.

 There are augmented map $\epsilon_X \colon H^{\ast}(X, \ZZ) \mapsto \ZZ$ and
 orientation classes $\mu_X \in H^{2n}(X)$, for $X \in \{M,N, M \# N, (S^{2n}
 \#_{T^k} S^{2n})\}$. Let $$R = \{(a, b) \in H^{\ast}(M) \times H^{\ast}(N) :
 \epsilon_M(a) = \epsilon_N(b)\}$$ and $ R (M, N) = R/\<(\mu_M, - \mu_N)\>$. Then
 the following cofibration $$ S^{2n -1} \to M \# N \to M \vee N$$ implies that
 if $i +j=2n$ with $i, j > 0$ then the cup-product of $H^i(M)$ with $H^j(N)$ in
 $H^{2n}(M \# N)$ is zero, but the product of $H^i(X)$ with $H^j(X)$ is the same
 as in the original manifold $ X \in \{M, N\}$.  Therefore, $H^{\ast}(M \# N, \ZZ)$
 is naturally isomorphic to $R(M, N)$. 

Note that $M \# N$ is an oriented manifold. Let
$$R' = \{(c, d) \in H^{\ast}(M \# N) \times H^{\ast}(S^{2n} \#_{T^k} S^{2n}) \colon
 \epsilon_{M\#N}(c) = \epsilon_{S^{2n} \#_{T^k} S^{2n}}(d)\}$$ and
\begin{equation}\label{eq:cohom_ring_conn_sum}
 R (M, N, T^k) = R'/\<(\mu_{M\#N}, - \mu_{S^{2n} \#_{T^k} S^{2n}})\>.
\end{equation} 
By similar arguments as in the previous paragraph,  $H^{\ast}(M \#_{T^k} N, \QQ)$
 is naturally isomorphic to $R(M, N, T^k)$. 
\end{proof}

\begin{remark}
The arguments in the proof of Theorem \ref{thm:cohom_conn_sum} is well-known. 
We represent it for the completeness. Explicit description of the cohomology ring
of a toric manifold is given in \cite{DJ, BP}.
\end{remark}

Now observing the algebra structure of $S^{2n}\#_{T^k}S^{2n}$, one may ask the following.
\begin{ques}
Is the torus manifold $S^{2n}\#_{T^k}S^{2n}$ homeomorphic to the connected
sum $\#_{i=1}^t (S^{m_i} \times S^{2n-m_i})$ of product spheres with the same cohomology
algebra for some $t$ and $0 < m_i < n+1$? 
\end{ques}

{\bf Acknowledgment:} The authors would like to thank Pacific Institute of Mathematical 
Sciences and University of Regina. 

\renewcommand{\refname}{References}
\bibliographystyle{alpha}
\bibliography{bibliography.bib}

\begin{thebibliography}{AMPZ17}

\bibitem[AMPZ17]{AMPZ}
A.~Ayzenberg, M.~Masuda, S.~Park, and H.~Zeng.
\newblock Cohomology of toric origami manifolds with acyclic proper faces.
\newblock {\em J. Symplectic Geom.}, 15(3):645--685, 2017.

\bibitem[Ayz16a]{Ayz1}
A.~Ayzenberg.
\newblock Homology cycles in manifolds with locally standard torus actions.
\newblock {\em Homology Homotopy Appl.}, 18(1):1--23, 2016.

\bibitem[Ayz16b]{Ayz2}
A.~Ayzenberg.
\newblock Locally standard torus actions and {$h'$}-numbers of simplicial
  posets.
\newblock {\em J. Math. Soc. Japan}, 68(4):1725--1745, 2016.

\bibitem[BP02]{BP}
V.~M. Buchstaber and T.~E. Panov.
\newblock {\em Torus actions and their applications in topology and
  combinatorics}, volume~24 of {\em University Lecture Series}.
\newblock American Mathematical Society, Providence, RI, 2002.

\bibitem[BR01]{BR}
V.~M. Buchstaber and N.~Ray.
\newblock Tangential structures on toric manifolds, and connected sums of
  polytopes.
\newblock {\em Internat. Math. Res. Notices}, (4):193--219, 2001.

\bibitem[Dav83]{Da}
M.~W. Davis.
\newblock Groups generated by reflections and aspherical manifolds not covered
  by {E}uclidean space.
\newblock {\em Ann. of Math. (2)}, 117(2):293--324, 1983.

\bibitem[DJ91]{DJ}
M.~W. Davis and T.~Januszkiewicz.
\newblock Convex polytopes, {C}oxeter orbifolds and torus actions.
\newblock {\em Duke Math. J.}, 62(2):417--451, 1991.

\bibitem[GK98]{GK}
M.~D. Grossberg and Y.~Karshon.
\newblock Equivariant index and the moment map for completely integrable torus
  actions.
\newblock {\em Adv. Math.}, 133(2):185--223, 1998.

\bibitem[HM03]{HM}
A.~Hattori and M.~Masuda.
\newblock Theory of multi-fans.
\newblock {\em Osaka J. Math.}, 40(1):1--68, 2003.

\bibitem[IFM13]{IFM}
H.~Ishida, Y.~Fukukawa, and M.~Masuda.
\newblock Topological toric manifolds.
\newblock {\em Mosc. Math. J.}, 13(1):57--98, 189--190, 2013.

\bibitem[Joy12]{Jo}
D.~Joyce.
\newblock On manifolds with corners.
\newblock In {\em Advances in geometric analysis}, volume~21 of {\em Adv. Lect.
  Math. (ALM)}, pages 225--258. Int. Press, Somerville, MA, 2012.

\bibitem[LS05]{LS}
P.~Lambrechts and D.~Stanley.
\newblock Algebraic models of {P}oincar\'e embeddings.
\newblock {\em Algebr. Geom. Topol.}, 5:135--182 (electronic), 2005.

\bibitem[MP06]{MP}
M.~Masuda and T.~Panov.
\newblock On the cohomology of torus manifolds.
\newblock {\em Osaka J. Math.}, 43(3):711--746, 2006.

\bibitem[PS15]{PS2}
M.~Poddar and S.~Sarkar.
\newblock A class of torus manifolds with nonconvex orbit space.
\newblock {\em Proc. Amer. Math. Soc.}, 143(4):1797--1811, 2015.

\bibitem[Yos11]{Yo}
T.~Yoshida.
\newblock Local torus actions modeled on the standard representation.
\newblock {\em Adv. Math.}, 227(5):1914--1955, 2011.

\end{thebibliography}

\vspace{1cm}

\vfill

\end{document}